\documentclass{article}
\usepackage[utf8]{inputenc}

\usepackage{amsmath}    
\usepackage{amssymb}    
\usepackage{hyperref}   
\usepackage{graphicx}   
\usepackage{listings}   
\usepackage{amsthm}     
\usepackage{amsfonts}
\usepackage{color}
\usepackage{verbatim}

\usepackage{tikz} 
\usetikzlibrary{shapes.geometric}
\tikzset{gon/.style={name=tmp,regular polygon,regular polygon sides=#1,minimum
size=10pt,inner sep=0pt},
polygon side/.style args={#1--#2}{
insert path={(tmp.corner #1)-- (tmp.corner #2)}}}
\newcommand{\FlagGraph}[3][]{\ifnum#2=2%
\tikz[baseline=(tmp1)]{\node[circle,inner sep=0.7pt,fill] (tmp1) at (0,0){};
\node[#1,circle,inner sep=0.7pt,fill] (tmp2) at (0,10pt){};
\ifx#3\empty%
\else
\draw[#1] (tmp1) -- (tmp2);
\fi}
\else%
\tikz[baseline=(tmp.south)]{\node[#1,gon=#2]{};
\foreach \X in {1,...,#2}{\fill (tmp.corner \X) circle (1pt);}
\draw[#1,polygon side/.list={#3}]}
\fi}

\usepackage{pgf,tikz}
\usepackage{xspace}
\xspaceaddexceptions{]\}}

\usepackage{chngcntr}

\theoremstyle{plain}
\newtheorem{theorem}{Theorem}

\newtheorem{corollary}[theorem]{Corollary}
\newtheorem{lemma}[theorem]{Lemma}

\theoremstyle{definition}
\newtheorem{definition}{Definition}

\DeclareMathOperator*{\inflate}{inflate}

\newcommand{\fix}[1]{\ifmmode{#1}\else{$#1$}\xspace\fi} 

\title{3-Symmetric Graphs}
\author{Sebastian Jeon\\
MIT
\and
Tanya Khovanova\\
MIT}
\date{}

\begin{document}

\maketitle

\begin{abstract}
An intuitive property of a random graph is that its subgraphs should also appear randomly distributed. We consider graphs whose subgraph densities exactly match their expected values. We call graphs with this property for all subgraphs with $k$ vertices to be $k$-symmetric. We discuss some properties and examples of such graphs. We construct 3-symmetric graphs and provide some statistics.
\end{abstract}

\section{Introduction}

The motivation for this paper starts with quasirandom sequences of graphs, permutations and other objects.

Any $k$ elements in a permutation of $n > k$ elements form a subpermutation of length $k$ with respect to relative values. Given a permutation of size $k$, we call its density in a given permutation of size $n$ the number of its occurrences as a subpermutation divided by the total number of subsets of size $k$.

Consider a sequence of permutations of growing length. This sequence is quasirandom if the densities of every permutation of length $k$ tend to $\frac{1}{k!}$ as the the lengths of permutations tend to infinity. 

In 1989 Chung, Graham and Wilson \cite{CGW} showed that if the density of 4-vertex subgraphs in a  large graph is asymptotically the same as in a random graph then this is true for every fixed subgraph. In 2013, Kr\'{a}l' and Pikhurko \cite{KP} proved a similar result for permutations. 

By this result, to show that the sequence is asymptotically random it is enough to show that the densities of 4-vertex subgraphs are random. For sequences that are not asymptotically random, it is interesting to study the densities of 3-vertex subgraphs.

In 2018 Khovanova and Zhang \cite{KZ} studied finite permutations that exhibit properties of random permutation. Namely, a permutation is called $k$-symmetric if every subpermutation of length $k$ has the same density. They showed that such permutations can only exist for lengths satisfying certain divisibility constraints, and constructed 3-symmetric permutations of small lengths. They also conjectured that there exists a 3-symmetric permutation of each admissible length.

In this paper we study finite graphs that exhibit properties of random graphs. We introduce the notion of a $k$-symmetric graph, which is parallel to the definition of a $k$-symmetric permutation: a graph is $k$-symmetric if densities of all subgraphs with $k$ vertices equals the expected density of these subgraphs in a random graph.

We provide some observations about $k$-symmetric graphs for any $k$. There is a natural constraint for the orders, that is the number of vertices, of $k$-symmetric graphs related to the divisibilities of binomial coefficients. We show that if such graphs exist then the smallest order (only considering the divisibility constraint) is a power of 2 whose exponent is $\binom{k}{2}+\nu_2(k)$. For example, the smallest possible order for 4-symmetric graphs is 256. As this is a very large number of vertices, we concentrate on 3-symmetric graphs in the rest of the paper. However, we show one more general result on $k$-symmetric graphs, that $k$-symmetricity implies $j$-symmetricity for $j < k$. 

For 3-symmetric graphs, the divisibility constraint implies that the smallest orders where they can exist are 8, 16, and 17. We find 3-symmetric graphs in all of these orders. 3-symmetric graphs were defined by Khovanova in \cite{K} and later Perkinson \cite{P} calculated that the number of such graphs of order 8 is 74. The next available order is 16. We found some 3-symmetric graphs, but working with graphs of this orders is non-trivial due to computational constraints. Thus rather than calculating the exact number of 3-symmetric graphs of order 16, we provide some statistics.

Khovanova and Zhang \cite{KZ} used an inflation procedure on permutations to build 3-symmetric permutations of larger sizes. This motivated us to study inflations of 3-symmetric graphs. For graphs, inflations do not work the same way as in permutations in that it does not preserve 3-symmetricity. But we show that inflating 3-symmetric graphs create 3-symmetric graphs asymptotically, i.e. the densities tend to their expected values.

In Section~\ref{sec:pr}, we give formal definitions of $k$-symmetric graphs, the objects of study. 

In Section~\ref{sec:tr}, we consider some general results on $k$ symmetric graphs. We formalize the divisibility constraint on the orders of $k$-symmetric graphs to be the condition $\nu_2(\binom{n}{k})\ge \binom{k}{2}$. Moreover, we prove that $k$-symmetric graphs are also $j$-symmetric for each $j$ less than $k$.

In Section~\ref{sec:di}, we define the inflation procedure and we provide the formulae for how densities behave under the inflation procedure. We show that the inflation of two 2-symmetric graphs is 2-symmetric, but that the analogous result for 3-symmetric graphs is not true. We prove that the inflation of a 3-symmetric graph $G$ into a 3-symmetric graph $H$ tends to be 3-symmetric when the order of $G$ tends to infinity.

In Section~\ref{sec:cr}, we provide examples of computer-generated 3-symmetric graphs of orders 16 and 17. We also give statistics on maximum clique and degree sizes for randomly sampled 3-symmetric graphs with 16 vertices.

\section{Preliminaries}\label{sec:pr}

\subsection{Defining 
\textit{k}-symmetric Graphs}

We want to translate the notion of $k$-symmetricity from permutations to graphs. $k$-symmetric permutations were introduced in \cite{KZ} and $k$-symmetric graphs in \cite{K}.
A $k$-\textit{symmetric} permutation is such that the densities of all permutations of length $k$ in it are the same. In particular, a 2-symmetric permutation has the same number of inversions and non-inversions. 

How do we create an analogous definition for graphs? We call a graph 2-symmetric if it has the same number of edges as non-edges. 

The above definition of a 2-symmetric graph is difficult to generalize. So we rephrase: a graph $G$ is 2-symmetric, if the density of any subgraph $H$ with 2 vertices in $G$ is the same as the expected density of $H$ in a random graph where the probability of an edge equals 1/2. This definition is easy to generalize: 

A graph $G$ is $k$-\textit{symmetric}, if the density of any subgraph $H$ with $k$ vertices in $G$ is the same as the expected density of $H$ in a random graph where the probability of an edge equals 1/2.

For the rest of the paper, we define $t(H,G)$ to be the density of graph $H$ in graph $G$.

\subsection{2-symmetric graphs}

We denote the density of edges in $G$ as $t\left(\FlagGraph{2}{1--2},G\right)$ and the density of non-edges as $t\left(\FlagGraph{2}{},G\right)$. By definition, a graph is 2-symmetric if and only if
\[t\left(\FlagGraph{2}{1--2},G\right) = t\left(\FlagGraph{2}{},G\right) =\frac{1}{2}.\]
The graphs with 0 or 1 vertices are trivially 2-symmetric. 2-symmetric graphs with 2 or 3 vertices do not exist. The simplest non-trivial examples are graphs with 4 vertices and three edges. There are 3 such graphs: a path, a star and complete graph $K_3$ with an isolated vertex. These graphs are depicted in Figure ~\ref{fig:2-sym4}. Note that the last two graphs are complements of each other and the first graph is self-complementary.
\begin{figure}[ht]
\centering
\scalebox{0.2} {
\includegraphics{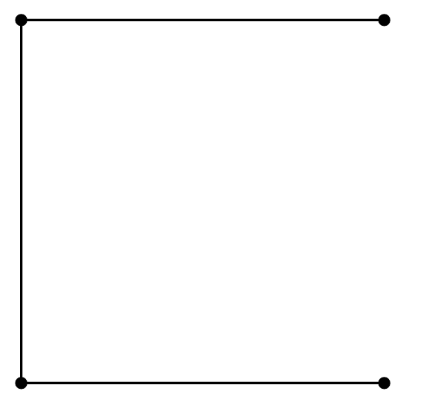} \hspace{4cm} \includegraphics{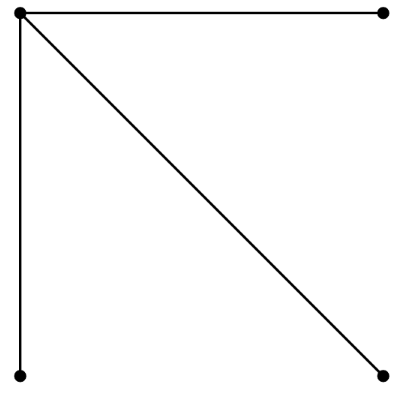} \hspace{4cm} \includegraphics{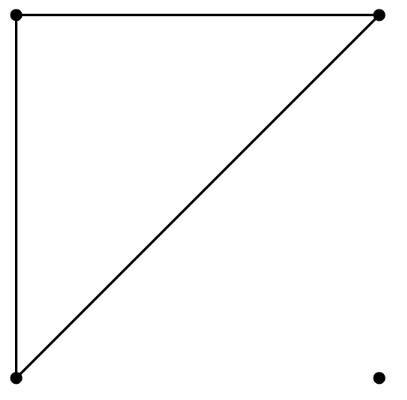}
}
\caption{The 2-symmetric graphs of order 4}
\label{fig:2-sym4}
\end{figure}
The number of 2-symmetric graphs with $n$ vertices is provided by sequence A218113 in the Online Encyclopedia of Integer Sequences \cite{OEIS}. The sequence, with the first index 1, starts as follows:
\[1,\ 0,\ 0,\ 3,\ 6,\ 0,\ 0,\ 1646,\ 34040,\ 0,\ 0,\ 16006173014,\ 4525920859198,\ \ldots .\]

\subsection{3-symmetric Graphs}

We denote the densities of subgraphs with 3 vertices in $G$ in the following manner. The density of the complete graph $K_3$ in $G$ as $t(\FlagGraph{3}{1--2,2--3,3--1}, G)$, the density of the path graph $P_3$ as $t(\FlagGraph{3}{1--2,1--3}, G)$, the density of the single edge with an isolated vertex as $t(\FlagGraph{3}{2--3}, G)$, and the density of the independent set on 3 vertices as $t(\FlagGraph{3}{}, G)$.

From the definition of 3-symmetric graphs, the densities of all four possible subgraphs with 3 vertices in a 3-symmetric graph should be as follows.

\begin{itemize}
    \item A complete graph with 3 vertices: $t(\FlagGraph{3}{1--2,2--3,3--1}, G) = \frac{1}{8}$,
    \item A path graph with 3 vertices: $t(\FlagGraph{3}{1--2,1--3}, G) = \frac{3}{8}$,
    \item A graph with 3 vertices and only one edge: $t(\FlagGraph{3}{2--3}, G)=\frac{3}{8}$,
    \item A graph with 3 isolated vertices: $t(\FlagGraph{3}{}, G)=\frac{1}{8}$.
\end{itemize}

The graphs with 0, 1 and 2 vertices are trivially 3-symmetric. As we show in the next section 3-symmetric graphs with 3 to 7 vertices do not exist. 
The first non-trivial case is $n = 8$. Figure~\ref{fig:3sym8} shows two 3-symmetric graphs. The first one is a wheel, and the second one is its complement.

\begin{figure}[htp!]
\centering
\includegraphics[scale=0.4]{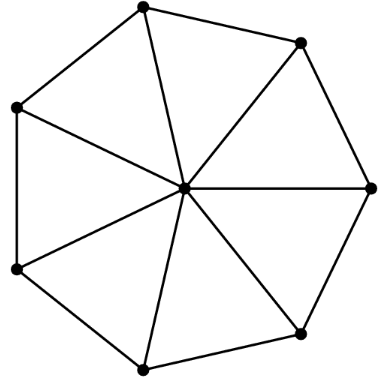} \quad \quad
\includegraphics[scale=0.4]{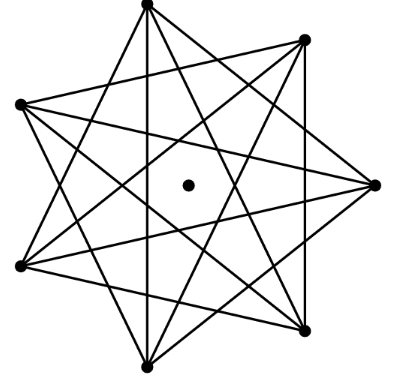}
\caption{Examples of 3-symmetric graphs with 8 vertices}
\label{fig:3sym8}
\end{figure}

Perkinson \cite{P} calculated that there are $74$ 3-symmetric graphs of order 8.

\section{\textit{k}-symmetric graphs}
\label{sec:tr}

\subsection{The restriction on the number of vertices}

If $n < k$, then the densities of all subgraphs of order $k$ are zero and the same. Such graphs are $k$-symmetric.

Suppose the number of vertices $n > k$. For a graph $G$ with $n$ vertices to be $k$-symmetric, we need $\binom{n}{k}$ to be divisible by $2^{\binom{k}{2}}$. This is because the density of a complete graph with $k$ vertices has to be $\frac{1}{2^{\binom{k}{2}}}$, which means the number of $k$-subgraphs of $G$ must be a multiple of that denominator. 

We call a number $n> k$ \textit{k-admissible} if $\binom{n}{k}$ is divisible by $2^{\binom{k}{2}}$. If $n> k$ is not $k$-admissible, then a $k$-symmetric graph with $n$ vertices does not exist.

In particular, for 2-symmetric graphs, $\binom{n}{2}$ must be even, which is equivalent to $n\equiv 0,1\bmod 4$.

By the above discussion, for a 3-symmetric graph, the number of vertices $n$ needs to be such that $\binom{n}{3}$ is divisible by 8. The sequence of numbers $n$ such that $\binom{n}{3}$ is divisible by 8 starts as:
\[1,\ 2,\ 8,\ 10,\ 16,\, 17,\ 18,\ 24,\ 26,\ 32,\ 33,\ 34,\ 40,\ 42,\ 48,\ 49,\ 50,\ 56,\ \ldots\]
These numbers are 0, 1, 2, 8, and 10 $\mod$ 16. This sequence is now sequence A329952 on the OEIS.



For 4-symmetric graphs we need $\binom{n}{4}$ to be divisible by $2^6$. So the minimum 4-admissible $n$ is $n=256$.

The smallest $k$-admissible numbers, starting from $k=2$ are given by the sequence 
\[4,\ 8,\ 256,\ 1024,\ 65536,\ 2097152\ \ldots.\]

This is the smallest $n$ such that $2^{k(k-1)/2}$ divides $\binom{n}{k}$. This sequence corresponds to the following powers of 2: 
\[2,\ 3,\ 8,\ 10,\ 16,\ 21,\ 31,\ \ldots.\]


This is now sequence A326714 in the OEIS \cite{OEIS}

We will prove that, as the sequence suggests, that the smallest $k$-admissible number is a power of 2 for all $k$. In what follows we denote 2-adic valuation of $n$ as $\nu_2(n)$.

\begin{lemma}
Given integers $k$ and $m$, such that $2^{m + \nu_2(k)} > k$, the smallest integer $n$ such that $\binom{n}{k}$ is divisible by $2^m$ is $2^{m+ \nu_2(k)}$.
\end{lemma}

\begin{proof}
The largest power of 2 that divides $\binom{n}{k}$ is the number of carries when summing up $n-k$ and $k$ in base 2. This number must be less than the number of digits of $n$, which we denote by $d$. Moreover, the last $\nu_2(k)$ digits of $k$ are zeros and do not contribute to the number of carries. Thus, the largest power of 2 that divides $\binom{n}{k}$ is less than $d- \nu_2(k)$. Hence, if $n < 2^{m+ \nu_2(k)}$, the largest power of 2 that divides $\binom{n}{k}$ is less than $m$.

On the other hand, if $n = 2^{m+ \nu_2(k)}$, the number of carries is exactly $m$.
\end{proof}

\begin{corollary}
The smallest $k$-admissible number is $2^{\binom{k}{2}+\nu_2(k)}$.
\end{corollary}

\subsection{$k$-symmetricity implies $j$-symmetricity for $j <k$}
In this section, we prove that a $k$-symmetric graph must be $j$-symmetric for $j<k$. This preservation of symmetricity property suggests that the definition of symmetricity is natural.

Recall that $t(H,G)$ is the density of the graph $H$ in the graph $G$.

\begin{theorem}
A non-trivial $k$-symmetric graph is $j$-symmetric for $j<k$.
\end{theorem}
\begin{proof}
It suffices to show that a nontrivial $k$-symmetric graph is $(k-1)$-symmet\-ric, as then induction would finish the rest. 

Let $G$ be $k$-symmetric, and now consider a particular graph $\widetilde{H}$ with $k-1$ vertices. We calculate the density of $\widetilde{H}$ in $G$ by calculating its density in subgraphs of $G$ of order $k$, as follows. 
\[t(\widetilde{H},G)=\mathbb{E}_{H\subseteq G, |H|=k}[t(\widetilde{H},H)].\]
Since $G$ is $k$-symmetric, we can instead take the expectation over the uniform distribution of $H$ over graphs on $k$ vertices. Thus each subgraph of $H$ of order $k-1$ is also uniformly distributed over graphs on $k-1$ vertices, meaning that $t(\widetilde{H},G)$ equals the probability that $\widetilde{H}$ is isomorphic to a uniformly chosen random graph on $k-1$ vertices. As this is true for all $\widetilde{H}$, we are done.

By induction, it follows that a $k$-symmetric graph is $m$-symmetric for $m<k.$
\end{proof}

\begin{corollary}
A $3$-symmetric graph is $2$-symmetric.
\end{corollary}

As 3-symmetric graphs are 2-symmetric, they can only exist for $n$ such that $\binom{n}{3}$ is divisible by 8 and $\binom{n}{2}$ is divisible by 2. Thus we keep the numbers from the previous sequence that are  $0, 1, \bmod\ 4$:
\[1,\ 8,\ 16,\ 17,\ 24,\ 32,\ 33,\ \ldots .\]

This sequence contains the numbers that are 0, 1, and 8 modulo 16.

Notice that if a number is $k$-admissible, it does not have to be $j$-admissible by $j< k$. For example, 10 is 3-admissible, but not 2-admissible. On the other hand, the smallest $k$-admissible number is $j$-admissible for any $j < k$. This is because $\binom{k}{2}+\nu_2(k)-\left(\binom{k-1}{2}+\nu_2(k-1)\right)=k+\nu_2(k)-\nu_2(k-1)\ge k-\log_2(k-1)>0$ for all $k$, so the sequence of smallest $k$-admissible numbers is strictly increasing.

\subsection{Self-complementary graphs}

\begin{definition}
A graph is \textit{self-complementary} if it is isomorphic to its complement (the graph formed by flipping each of its edges). 
\end{definition}

The density for a graph $H$ in a self-complementary graph is equal to the density of its complement $H'$. That means for a self-complementary graph $G$ to be $k$-symmetric, it is enough for the densities of $k$-subgraphs with not more than $\frac{k(k-1)}{2}$ edges to provide the correct densities.

Applying this to the case where $H$ is an edge, we see that self-comeplementa\-ry graphs are 2-symmetric. Also, the densities of a 3-clique and a 3-vertex graph with all isolated vertices are the same.  Also, a self-complementary graph has the same density for a 3-vertex graph with 1 edge and 2 edges. Thus, a self-complementary graph is 3-symmetric if and only if the density of the clique is 1/8.

There are 10 self-complementary graphs of order 8 \cite{HP}. Unfortunately, none of them are 3-symmetric. But self-complementary graphs might provide examples of 3-symmetric graphs of higher orders.

Self-complementary graphs exist in the same orders as 2-symmetric graphs. Therefore, they exist in all orders where a $k$-symmetric graph might exist. 

The sequence A000171 in the OEIS \cite{OEIS} describes the number of self-comple\-mentary graphs with $n$ nodes. It starts as: 
\[1,\ 0,\ 0,\ 1, 2,\ 0,\ 0,\ 10,\ 36,\ 0,\ 0,\ 720,\ 5600,\ 0, \ldots.\] 

\section{Densities and Inflation}\label{sec:di}

We now discuss possible approaches for constructing larger 3-symmetric graphs. One such approach is to take two 3-symmetric graphs and combine them to obtain a larger one. One possible mechanism for doing so is known as the lexicographic product of graphs introduced by Hausdorff in 1914 \cite{H}. However, due to analogous notions introduced in \cite{KZ}, we will refer to this operation as an inflation.
For graphs $G$ and $H,$ define the \textit{inflation} of $G$ with respect to $H$ as the graph with $|G||H|$ vertices where:
\begin{itemize}
    \item Each vertex in $G$ becomes a graph isomorphic to $H$, and
    \item If $H_i$ and $H_j$ are the graphs that correspond to adjacent nodes $i$ and $j$ in $G,$ each vertex in $H_i$ becomes adjacent to each vertex in $H_j.$
\end{itemize}
We denote the inflation of $G$ with respect to $H$ as $\inflate(G,H)$.

Figure~\ref{fig:inf} provides an example of inflation, where $H$ is a star graph $S_4$ and $G$ is a path graph $P_4$.

\begin{figure}[ht]
    \centering
    \scalebox{0.7} {
    \includegraphics{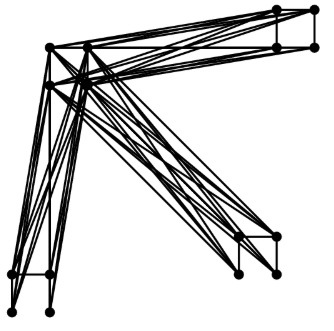}
    }
        \caption{The graph $\inflate\left(\FlagGraph{4}{2--1,2--3,2--4},\FlagGraph{4}{4--1,1--2,2--3}\right)$.}
        \label{fig:inf}
\end{figure}

\subsection{Densities in inflated graphs}
The number of edges in an inflation graph $\inflate(G,H)$ can be expressed through the number of vertices and edges in $G$ and $H$. The formula is well-known. Translated to densities we get the following lemma which describes how the density of edges behaves with respect to an inflation.

\begin{lemma}
Given graphs $H$ and $G$, the density of edges in $\inflate(G,H)$ is given by the following formula: \[t\left(\FlagGraph{2}{1--2},\inflate(G,H)\right)=\frac{|G|\tbinom{|H|}{2}t\left(\FlagGraph{2}{1--2},H\right)+\tbinom{|G|}{2}t\left(\FlagGraph{2}{1--2},G\right)|H|^2}{\tbinom{|G||H|}{2}}.\]
\label{edgeDensity}
\end{lemma}

We can express the density of a particular graph $G'$ with 3 vertices in $\inflate(G,H)$ through the density of $G'$ in $G$ and $H$ and the densities of edges in $G$ and $H$.

\begin{lemma} Given graphs $H$ and $G$, the density of $K_3$ in $\inflate(G,H)$ is given by the following formula:
\begin{multline*}
t\left(\FlagGraph{3}{1--2,1--3,2--3}, \inflate(G,H)\right)= \\
\frac{|G|t\left(\FlagGraph{3}{1--2,1--3,2--3}, H\right)\tbinom{|H|}{3}+2\tbinom{|G|}{2}t\left(\FlagGraph{2}{1--2},G\right)\tbinom{|H|}{2}|H|t\left(\FlagGraph{2}{1--2},H\right)+\tbinom{|G|}{3}t\left(\FlagGraph{3}{1--2,1--3,2--3},G\right)|H|^3}{\tbinom{|G||H|}{3}}.
\end{multline*}
\end{lemma}
\begin{proof}
We do casework on the distribution of the vertices of $K_3$ across the copies of $H$ in $\inflate(G,H)$. 

In the case where the three vertices all belong in one copy of $H$, there are $|G|$ copies of $H$ to choose from, each of which has $t(\FlagGraph{3}{1--2,1--3,2--3}, H)\tbinom{|H|}{3}$ triangles. This corresponds to the first term in the numerator. 

If two vertices are from one copy of $H$ and one is from a different copy, there are $2\binom{|G|}{2}t\left(\FlagGraph{2}{1--2}\right)$ choices of the ordered copies of $H$ that have edges between them. Now the two vertices that are in the same copy of $H$ must have an edge between them, so there are $\binom{|H|}{2}t\left(\FlagGraph{2}{1--2}, H\right)$ choices for these two vertices, and $|H|$ choices for the third vertex. 

Finally, when the three vertices are in different copies of $H$, there are $\binom{|G|}{3}t\left(\FlagGraph{3}{1--2,2--3,3--1},H\right)|H|^3$ sets of vertices that work.

Dividing the total by $\binom{|G||H|}{3}$ gives the desired density.
\end{proof}

We prove an analogous theorem for $P_3$:
\begin{lemma}
 Given graphs $H$ and $G$, the density of $P_3$ in $\inflate(G,H)$ is given by the following formula:
\begin{multline*}
t\left(\FlagGraph{3}{1--2,1--3},\inflate(G,H)\right)= \\
\frac{|G|t(\FlagGraph{3}{1--2,1--3},H)\tbinom{|H|}{3}+2\tbinom{|G|}{2}t(\FlagGraph{2}{1--2},G)\tbinom{|H|}{2}|H|t(\FlagGraph{2}{},H)+\tbinom{|G|}{3}t(\FlagGraph{3}{1--2,1--3},G)|H|^3}{\tbinom{|G||H|}{3}}.
\end{multline*}
\end{lemma}

\begin{proof}
We do casework on the distribution of the vertices of $P_3$ across the copies of $H$ in $\inflate(G,H)$. 

In the case where the three vertices all belong in one copy of $H$, there are $|G|$ copies of $H$ to choose from, each of which has $t(\FlagGraph{3}{1--2,1--3}, H)\tbinom{|H|}{3}$ copies of $P_3$. This corresponds to the first term in the numerator. 

If two vertices are from an $i$-th copy of $H$ and one is from a $j$-th copy of $H$, the two vertices from the copy of $H$ must be non-adjacent, but the vertices $i$ and $j$ in $G$ must be adjacent. Thus the number of $P_3$ subgraphs is the same as the number of ordered pairs of edges, one from $h$ and the other one from $G$. We get the total of  $2\binom{|G|}{2}t\left(\FlagGraph{2}{1--2}\right)$.

Finally, when the three vertices are in different copies of $H$, there are $\binom{|G|}{3}t\left(\FlagGraph{3}{1--2,1--3},H\right)|H|^3$ sets of vertices that work.

Dividing the total by $\binom{|G||H|}{3}$ gives the desired density.
\end{proof}

By considering swapping edges with non-edges, we can get formulae for densities of the other two subgraphs on three vertices. They are stated in the following two lemmas.

\begin{lemma} Given graphs $H$ and $G$, the density of the three isolated vertices in $\inflate(G,H)$ is given by the following formula:
\begin{multline*}
t\left(\FlagGraph{3}{}, \inflate(G,H)\right)= \\
\frac{|G|t\left(\FlagGraph{3}{}, H\right)\tbinom{|H|}{3}+2\tbinom{|G|}{2}t\left(\FlagGraph{2}{},G\right)\tbinom{|H|}{2}|H|t\left(\FlagGraph{2}{},H\right)+\tbinom{|G|}{3}t\left(\FlagGraph{3}{},G\right)|H|^3}{\tbinom{|G||H|}{3}}.
\end{multline*}
\end{lemma}

\begin{lemma}
 Given graphs $H$ and $G$, the density of the complement of $P_3$ in $\inflate(G,H)$ is given by the following formula:
\begin{multline*}
t\left(\FlagGraph{3}{2--3},\inflate(G,H)\right)= \\
\frac{|G|t(\FlagGraph{3}{2--3},H)\tbinom{|H|}{3}+2\tbinom{|G|}{2}t(\FlagGraph{2}{},G)\tbinom{|H|}{2}|H|t(\FlagGraph{2}{1--2},H)+\tbinom{|G|}{3}t(\FlagGraph{3}{2--3},G)|H|^3}{\tbinom{|G||H|}{3}}.
\end{multline*}
\end{lemma}

\subsection{2-symmetric graphs}

We are interested in 2-symmetric graphs and can deduce the following corollary from Lemma~\ref{edgeDensity}.

\begin{corollary}
If $H$ and $G$ are 2-symmetric graphs, then $\inflate(G,H)$ is also 2-symmetric.
\end{corollary}

\begin{proof}
Assume $t\left(\FlagGraph{2}{1--2},G\right)=t\left(\FlagGraph{2}{1--2},H\right)=\frac{1}{2}.$ Also let $x=|G|,y=|H|.$ Then 
\[t\left(\FlagGraph{2}{1--2},\inflate(G,H)\right)=\frac{\frac{xy(y-1)}{4}+\frac{x(x-1)y^2}{4}}{\frac{xy(xy-1)}{2}}=\frac{y-1+y(x-1)}{2(xy-1)}=\frac{1}{2},\]
as desired.
\end{proof}

For example, in Figure~\ref{fig:inf}, both graphs in the inflation are 2-symmetric, so the graph shown in the figure will also be 2-symmetric.

If graphs $G$ and $H$ are 2-symmetric the formulae for densities in their inflation simplifies. Moreover the resulting formula is the same of all the graphs with 3 vertices.

\begin{lemma}
If If graphs $G$ and $H$ are 2-symmetric, and $S$ is a graph with 3 vertices then the density of $S$ in $\inflate(G,H)$ is provided by the following formula:
\[t\left(S,\inflate(G,H)\right)= 
\frac{|G|t(S,H)\tbinom{|H|}{3}+\frac{1}{2}\tbinom{|G|}{2}\tbinom{|H|}{2}|H|+\tbinom{|G|}{3}t(S,G)|H|^3}{\tbinom{|G||H|}{3}}.\]
\label{threeinf}
\end{lemma}

\subsection{3-symmetric graphs}

If $G$ and $H$ are 3-symmetric, one might expect $\inflate(G,H)$ can be as well. The reason for this expectation is that the inflation of two 3-symmetric permutations can be a 3-symmetric permutation under certain divisibility conditions \cite{KZ}. However, this is not the case.

\begin{corollary}
If $G$ and $H$ are 3-symmetric graphs with more than one vertex, then  $\inflate(G,H)$ is not 3 symmetric.
\end{corollary}

\begin{proof}
By plugging in the density of $K_3$ as $\frac{1}{8}$ in both $G$ and $H$, and dividing the numerator and denominator by $\frac{|G||H|}{6}$ we get
$$\frac{\tfrac{1}{8}(|H|-1)(|H|-2)+\tfrac{3}{4}(|G|-1)|H|(|H|-1)+\tfrac{1}{8}(|G|-1)(|G|-2)|H|^2}{(|G||H|-1)(|G||H|-2)}.$$

After simplifying we get
$$\frac{-3|H|^2+3|H|+2+3|G||H|^2-6|G||H|+|G|^2|H|^2}{8(|G||H|-1)(|G||H|-2)}.$$
Subtracting $\frac{1}{8}$, we get
$$\frac{\tfrac{1}{8}(-3|H|^2+3|H|+3|G||H|^2-3|G||H|)}{(|G||H|-1)(|G||H|-2)}=\frac{\tfrac{1}{8}(3|H|(|H|-1)(|G|-1))}{(|G||H|-1)(|G||H|-2)}\geq 0.$$
Thus the density of $K_3$ in $\inflate(G,H)$ is not $\frac{1}{8}$.
\end{proof}

We call a graph $G$ \textit{almost-3-symmetric} if the following three conditions hold:
\begin{itemize}
    \item $G$ is 2-symmetric,
    \item $t(\FlagGraph{3}{1--2,2--3,3--1},G) = t(\FlagGraph{3}{},G)$,
    \item $t(\FlagGraph{3}{1--2,1--3},G)=t(\FlagGraph{3}{2--3},G)$.
\end{itemize}

\begin{lemma}
Any two of the conditions for almost-3-symmetric graph imply the third.
\end{lemma}
\begin{proof}
The first condition is equivalent to
\[3t\left(\FlagGraph{3}{1--2,2--3,3--1},G\right)+2t\left(\FlagGraph{3}{1--2,1--3},G\right)+t\left(\FlagGraph{3}{2--3},G\right)=3t\left(\FlagGraph{2}{1--2},G\right)=\frac{3}{2}\] upon counting the number of edges contributed by each subgraph of $G$ on three vertices.
The second condition is $$t\left(\FlagGraph{3}{1--2,2--3,3--1},G\right)=t\left(\FlagGraph{3}{},G\right),$$ and the third is $$t\left(\FlagGraph{3}{1--2,3--1},G\right)=t\left(\FlagGraph{3}{2--3},G\right).$$ Furthermore, there is the general condition $$t\left(\FlagGraph{3}{1--2,2--3,3--1},G\right)+t\left(\FlagGraph{3}{1--2,1--3},G\right)+t\left(\FlagGraph{3}{2--3},G\right)+t\left(\FlagGraph{3}{},G\right)=1.$$ Since these conditions are linearly dependent, it follows that any two conditions imply the third.
\end{proof}

By definition a 3-symmetric graph is almost-3-symmetric. Also, a self-complementary graph is almost-3-symmetric.

\begin{theorem}
If $G$ and $H$ are almost-3-symmetric, then $\inflate(G,H)$ is also almost-3-symmetric.
\end{theorem}

\begin{proof}
For a graph $S$ on three vertices, let $\overline{S}$ be its complement. We need to prove $t(S,\inflate(G,H))=t(\overline{S},\inflate(G,H))$. By Lemma \ref{threeinf}, it suffices to show
\[|G|t(S,H)\binom{|H|}{3}+\binom{|G|}{3}t(S,G)|H^3|=|G|t(\overline{S},H)\binom{|H|}{3}+\binom{|G|}{3}t(\overline{S},G)|H^3|\]
which follows from the assumption that $t(S,G)=t(\overline{S},G)$ and $t(S,H)=t(\overline{S},H)$.
\end{proof}

What is a potential number of vertices $n$ for almost-3-symmetric graph? It has to be 2-symmetric, that is $n$ has remainder 0 or 1 when divided by 4, The other condition is that $\binom{n}{3}$ should be divisible by 2. This is true for $n$ that has remainder 0 or 1 when divided by 4. That means, almost-3-symmetric graphs might exists with the same number of vertices that 2-symmetric graphs exist.

For example, out of four 2-symmetric graphs with 4 vertices, only $P_4$ is almost-3-symmetric.

We tried to inflate almost-3-symmetric graphs and check whether the result is 3-symmetric for small almost-3-symmetric graphs, but we could not find any such examples with our calculations.

\subsection{Asymptotics}
For all of the subgraphs, it is clear that the last terms in the formulae are dominating as $|G|\to \infty$. With this, we have found that the densities of all of the subgraphs tend to their expected densities in the limit case. We formalize this statement in the following theorem.
\begin{theorem}
Let $G_1,G_2,\ldots, G_n, \ldots$ be 3-symmetric graphs whose orders go to $\infty$, and $H$ also be 3-symmetric. Then the densities of any 3-subgraph into the inflation of $H$ into $G_i$ will tend to their expected density in a random graph.
\end{theorem}

\begin{proof}
Since the $G_i$ and $H$ are 3-symmetric, we have $t(\FlagGraph{3}{1--2,2--3,3--1},G_i)=t(\FlagGraph{3}{1--2,2--3,3--1},H)=\frac{1}{8}$, $t(\FlagGraph{3}{1--2,1--3},G_i)=t(\FlagGraph{3}{1--2,1--3},H)=\frac{1}{8}$. Now with the formulas above, we have the asymptotic formulas \[t(\FlagGraph{3}{1--2,2--3,3--1}, \inflate(G_i,H))=\frac{|G_i|^3|H|^3\cdot \frac{1}{6}\cdot\frac{1}{8}+\mathcal{O}(|G_i|^2)}{\frac{1}{6}|G_i|^3|H|^3+\mathcal{O}(|G_i|)^2}\to\frac{1}{8}\]
as $|G_i|$ gets large. A similar asymptotic formula holds for the path:
\[t(\FlagGraph{3}{1--2,1--3},\inflate(G_i,H))=\frac{|G_i|^3|H|^3\cdot\frac{1}{6}\cdot\frac{3}{8}}{\frac{1}{6}|G_i|^3|H|^3+\mathcal{O}(|G_i|^2)}\to\frac{3}{8}.\]
By symmetry, the analogous statements for the other two subgraphs hold. This proves the theorem.
\end{proof}

\section{Constructing 3-symmetric graphs}\label{sec:cr}

With the aid of a computer, we found examples of 3-symmetric graphs of all feasible orders up to and including 40. In the following section, we explicitly describe graphs of orders 16 and 17.
\subsection{Orders 16 and 17}


As we showed before the next orders of a graph that could be 3-symmetric is 16 and 17. 

We found such graphs by randomly sampling 2-symmetric graphs. Figure~\ref{size16} shows an example of a 3-symmetric graph of order 16. Its adjacency matrix is as follows:
\[\left[ \begin{array}{@{}*{16}{c}@{}}
     0 & 0 & 1 & 0 & 0 & 0 & 0 & 0 & 0 & 0 & 0 & 0 & 1 & 1 & 0 & 0 \\
0 & 0 & 1 & 1 & 0 & 1 & 0 & 0 & 1 & 0 & 1 & 0 & 1 & 0 & 0 & 1 \\
1 & 1 & 0 & 1 & 1 & 0 & 0 & 0 & 1 & 0 & 0 & 1 & 0 & 1 & 1 & 1 \\
0 & 1 & 1 & 0 & 1 & 1 & 0 & 0 & 0 & 0 & 1 & 0 & 1 & 1 & 1 & 0 \\
0 & 0 & 1 & 1 & 0 & 1 & 1 & 0 & 0 & 0 & 0 & 0 & 1 & 0 & 1 & 0 \\
0 & 1 & 0 & 1 & 1 & 0 & 0 & 1 & 1 & 1 & 1 & 1 & 1 & 0 & 0 & 1 \\
0 & 0 & 0 & 0 & 1 & 0 & 0 & 0 & 1 & 1 & 1 & 0 & 1 & 1 & 1 & 1 \\
0 & 0 & 0 & 0 & 0 & 1 & 0 & 0 & 0 & 0 & 1 & 1 & 1 & 0 & 0 & 1 \\
0 & 1 & 1 & 0 & 0 & 1 & 1 & 0 & 0 & 0 & 0 & 0 & 0 & 0 & 1 & 1 \\
0 & 0 & 0 & 0 & 0 & 1 & 1 & 0 & 0 & 0 & 1 & 1 & 0 & 0 & 1 & 1 \\
0 & 1 & 0 & 1 & 0 & 1 & 1 & 1 & 0 & 1 & 0 & 0 & 0 & 1 & 1 & 1 \\
0 & 0 & 1 & 0 & 0 & 1 & 0 & 1 & 0 & 1 & 0 & 0 & 1 & 1 & 0 & 1 \\
1 & 1 & 0 & 1 & 1 & 1 & 1 & 1 & 0 & 0 & 0 & 1 & 0 & 0 & 0 & 0 \\
1 & 0 & 1 & 1 & 0 & 0 & 1 & 0 & 0 & 0 & 1 & 1 & 0 & 0 & 1 & 1 \\
0 & 0 & 1 & 1 & 1 & 0 & 1 & 0 & 1 & 1 & 1 & 0 & 0 & 1 & 0 & 1 \\
0 & 1 & 1 & 0 & 0 & 1 & 1 & 1 & 1 & 1 & 1 & 1 & 0 & 1 & 1 & 0 
\end{array} \right]\]
\begin{figure}[ht]
\centering
\includegraphics[scale=0.5]{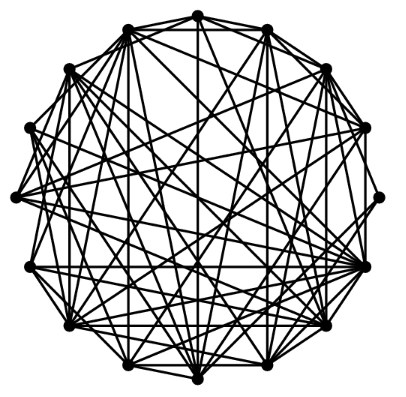}
\caption{A 3-symmetric graph of order 16}
\label{size16}
\end{figure}

We also found 3-symmetric graphs of order 17. Its adjacency matrix is shown below and its picture is shown in Figure~\ref{fig:size17}.

\[\left[ \begin{array}{@{}*{17}{c}@{}}
0 & 1 & 0 & 0 & 1 & 0 & 1 & 0 & 1 & 0 & 0 & 0 & 1 & 0 & 0 & 0 & 1 \\
1 & 0 & 0 & 0 & 0 & 1 & 0 & 0 & 0 & 0 & 1 & 1 & 0 & 1 & 0 & 1 & 0 \\
0 & 0 & 0 & 0 & 0 & 1 & 1 & 0 & 1 & 1 & 0 & 1 & 0 & 1 & 0 & 1 & 0 \\
0 & 0 & 0 & 0 & 1 & 0 & 1 & 0 & 0 & 1 & 1 & 0 & 0 & 1 & 0 & 0 & 1 \\
1 & 0 & 0 & 1 & 0 & 1 & 1 & 0 & 0 & 1 & 1 & 0 & 1 & 0 & 1 & 0 & 0 \\
0 & 1 & 1 & 0 & 1 & 0 & 1 & 0 & 0 & 0 & 1 & 1 & 0 & 0 & 0 & 1 & 1 \\
1 & 0 & 1 & 1 & 1 & 1 & 0 & 1 & 1 & 1 & 1 & 1 & 1 & 0 & 0 & 0 & 1 \\
0 & 0 & 0 & 0 & 0 & 0 & 1 & 0 & 0 & 1 & 1 & 1 & 0 & 1 & 1 & 1 & 1 \\
1 & 0 & 1 & 0 & 0 & 0 & 1 & 0 & 0 & 1 & 0 & 1 & 1 & 0 & 1 & 1 & 0 \\
0 & 0 & 1 & 1 & 1 & 0 & 1 & 1 & 1 & 0 & 1 & 0 & 1 & 0 & 0 & 1 & 0 \\
0 & 1 & 0 & 1 & 1 & 1 & 1 & 1 & 0 & 1 & 0 & 1 & 0 & 0 & 1 & 1 & 0 \\
0 & 1 & 1 & 0 & 0 & 1 & 1 & 1 & 1 & 0 & 1 & 0 & 1 & 0 & 0 & 1 & 1 \\
1 & 0 & 0 & 0 & 1 & 0 & 1 & 0 & 1 & 1 & 0 & 1 & 0 & 0 & 1 & 1 & 0 \\
0 & 1 & 1 & 1 & 0 & 0 & 0 & 1 & 0 & 0 & 0 & 0 & 0 & 0 & 0 & 1 & 0 \\
0 & 0 & 0 & 0 & 1 & 0 & 0 & 1 & 1 & 0 & 1 & 0 & 1 & 0 & 0 & 1 & 0 \\
0 & 1 & 1 & 0 & 0 & 1 & 0 & 1 & 1 & 1 & 1 & 1 & 1 & 1 & 1 & 0 & 1 \\
1 & 0 & 0 & 1 & 0 & 1 & 1 & 1 & 0 & 0 & 0 & 1 & 0 & 0 & 0 & 1 & 0 
\end{array} \right]\]

\begin{figure}[ht]
\begin{center}
\includegraphics[scale=0.5]{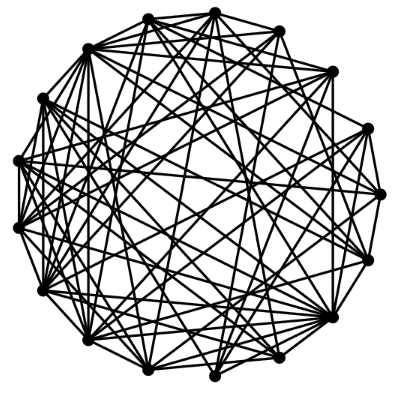} 
\end{center}
\caption{A 3-symmetric graph of order 17.}
\label{fig:size17}
\end{figure}

\subsection{Computational Results}

We randomly sampled a 2-symmetric graph of order 16 and checked whether it was 3-symmetric. This procedure allowed us to generate adjacency matrices for 3-symmetric graphs of order 16. Across 10000 trials, the probability that a random 2-symmetric graph is 3-symmetric was approximately $4.51\%$. Given that the number of 2-symmetric graphs of order 16 is  4648429222263945620900, the estimated number of 3-symmetric graphs of order 16 is $\approx 2.09\times 10^{20}$.

We would like to add that the number of 2-symmetric graphs of order 8 is 1646, and the number of 3-symmetric graphs os the same order is 74 \cite{P}. Thus, the percentage is about $4.5\%$.

We provide more statistics on 500 different 3-symmetric graphs of order 16 generated with the above procedure.

First we look for maximum clique sizes. On one hand, a 3-symmetric graph has to contain $K_3$. Thus, the maximum clique size cannot be less than 3. On the other hand, since a clique of order $9$ has $84$ triangles and a 3-symmetric graph of order $16$ has $70$ triangles, it follows that the maximum possible clique size is $8$. 
The following table shows maximum clique sizes in our sample.
\begin{center}
\begin{tabular}{|c|c|}
\hline Max Clique & Frequency
\\ \hline 4 & 41
\\ \hline 5 & 436
\\ \hline 6 & 23
\\ \hline
\end{tabular}
\end{center}

Similarly, we look at the maximum degrees of the graphs we found. Since the average degree in a 2-symmetric graphs of order 16, and thus 3-symmetric graphs of order 16, is $7.5$, the maximum degree is at least 8.  
The following table shows maximum degrees in our sample.
\begin{center}
\begin{tabular}{|c|c|}
\hline 
Max Degree & Frequency
\\ \hline 9 & 1
\\ \hline 10 & 115
\\ \hline 11 & 260
\\ \hline 12 & 109
\\ \hline 13 & 14
\\ \hline 14 & 1
\\ \hline
\end{tabular}
\end{center}
In particular, we only found one graph where the max degree is 9. The adjacency matrix is shown below.

\[\left[ \begin{array}{@{}*{16}{c}@{}}
    0 & 0 & 1 & 0 & 1 & 1 & 0 & 1 & 0 & 1 & 1 & 1 & 1 & 0 & 0 & 0 \\
    0 & 0 & 0 & 1 & 0 & 1 & 1 & 1 & 0 & 1 & 1 & 1 & 0 & 1 & 1 & 0 \\
    1 & 0 & 0 & 1 & 1 & 0 & 1 & 0 & 0 & 1 & 0 & 1 & 1 & 1 & 0 & 1 \\
    0 & 1 & 1 & 0 & 1 & 1 & 0 & 1 & 1 & 0 & 0 & 1 & 0 & 1 & 1 & 0 \\
    1 & 0 & 1 & 1 & 0 & 0 & 1 & 0 & 1 & 1 & 0 & 1 & 0 & 0 & 1 & 1 \\
    1 & 1 & 0 & 1 & 0 & 0 & 1 & 0 & 0 & 1 & 1 & 0 & 0 & 1 & 1 & 0 \\
    0 & 1 & 1 & 0 & 1 & 1 & 0 & 0 & 0 & 1 & 1 & 0 & 1 & 1 & 1 & 0 \\
    1 & 1 & 0 & 1 & 0 & 0 & 0 & 0 & 0 & 1 & 0 & 0 & 0 & 0 & 0 & 0 \\ 
    0 & 0 & 0 & 1 & 1 & 0 & 0 & 0 & 0 & 1 & 0 & 1 & 0 & 1 & 0 & 1 \\
    1 & 1 & 1 & 0 & 1 & 1 & 1 & 1 & 1 & 0 & 0 & 0 & 0 & 1 & 0 & 0 \\
    1 & 1 & 0 & 0 & 0 & 1 & 1 & 0 & 0 & 0 & 0 & 1 & 0 & 0 & 1 & 1 \\
    1 & 1 & 1 & 1 & 1 & 0 & 0 & 0 & 1 & 0 & 1 & 0 & 0 & 1 & 1 & 0 \\
    1 & 0 & 1 & 0 & 0 & 0 & 1 & 0 & 0 & 0 & 0 & 0 & 0 & 0 & 0 & 0 \\
    0 & 1 & 1 & 1 & 0 & 1 & 1 & 0 & 1 & 1 & 0 & 1 & 0 & 0 & 0 & 1 \\
    0 & 1 & 0 & 1 & 1 & 1 & 1 & 0 & 0 & 0 & 1 & 1 & 0 & 0 & 0 & 0 \\
    0 & 0 & 1 & 0 & 1 & 0 & 0 & 0 & 1 & 0 & 1 & 0 & 0 & 1 & 0 & 0 

\end{array} \right]\]

It is too computationally challenging to use this process  to find 4-symmetric graphs, as such graphs have order at least 256.

\section{Acknowledgements}

We are grateful to the MIT PRIMES program for giving us the opportunity to do this research. We are also grateful to David Perkinson for sharing his calculations with us and Yongyi Chen for reviewing the paper.

\end{document}